\documentclass[11pt,reqno]{amsart}
\usepackage[T1]{fontenc}
\usepackage{lmodern}
\usepackage[english]{babel}
\usepackage{graphicx}
\usepackage{amsmath}
\usepackage{amssymb}
\usepackage{mathrsfs}
\usepackage{tikz}
\usetikzlibrary{matrix}
\usepackage[all]{xy}
\usepackage{tikz-cd}
\usepackage{bookmark}
\usepackage[numbers,sort&compress]{natbib}
\usepackage{xcolor}
\usepackage{slashed}
\usepackage{hyperref}
\usepackage{cancel}
\numberwithin{equation}{section}
\usepackage{fancyhdr}
\usepackage{amsthm}
\usepackage{braket}
\usepackage{feynmf}
\usepackage[a4paper]{geometry}
\newcommand\Item[1][]{%
  \ifx\relax#1\relax  \item \else \item[#1] \fi
  \abovedisplayskip=0pt\abovedisplayshortskip=0pt~\vspace*{-\baselineskip}}
\begin{document}

\title[The Cohomology of the Virasoro algebra]{The Low-Dimensional Algebraic Cohomology of the Virasoro Algebra}
\author[Jill Ecker]{Jill Ecker}
\address{University of Luxembourg, Faculty of Science, Technology and Communication,
Campus Belval, Maison du Nombre, 6, avenue de la Fonte, L-4364 Esch-sur-Alzette, Luxembourg}
\curraddr{}
\email{jill.ecker@uni.lu}
\author[Martin Schlichenmaier]{Martin Schlichenmaier}
\address{University of Luxembourg, Faculty of Science, Technology and Communication,
Campus Belval, Maison du Nombre, 6, avenue de la Fonte, L-4364 Esch-sur-Alzette, Luxembourg}
\curraddr{}
\email{martin.schlichenmaier@uni.lu}
\thanks{Partial  support by the
Internal Research Project  GEOMQ15,  University of Luxembourg,
and
by the OPEN programme  of the Fonds National de la Recherche
(FNR), Luxembourg,  project QUANTMOD O13/570706
is gratefully acknowledged.}
\subjclass[2000]{Primary: 17B56; Secondary: 17B68, 17B65, 17B66, 14D15, 
81R10, 81T40}

\keywords{Witt algebra; Virasoro algebra; Lie algebra cohomology; Deformations of algebras; conformal field theory}
\date{16.05.2018}
\begin{abstract}
The main aim of this article is to prove the one-dimensionality of the third algebraic cohomology of the Virasoro algebra with values in the adjoint module. We announced this result in a previous publication with only a sketch of the proof. The detailed proof is provided in the present article. We also show that the third algebraic cohomology of the Witt and the Virasoro algebra with values in the trivial module is one-dimensional.  We consider purely algebraic cohomology, i.e. our results are independent of any topology chosen. \\
The vanishing of the third algebraic cohomology of the Witt algebra with values in the adjoint module has already been proven by Ecker and Schlichenmaier. 
\end{abstract}

\maketitle

\section{Introduction}
The Witt algebra and its universal central extension, the Virasoro algebra, are two of the most important infinite-dimensional Lie algebras, as they have many applications both in mathematics and theoretical physics. The Virasoro algebra is omnipresent in string theory, where it is attached to physical observables such as the mass spectrum. In two-dimensional conformal field theory, the Virasoro algebra is of outermost importance.\\
The cohomology of Lie algebras and the low-dimensional cohomology in particular, has numerous interpretations in terms of known objects such as invariants, outer derivations, extensions, deformations and obstructions, as well as crossed modules, see e.g. Gerstenhaber \cite{MR0161898,MR0171807,MR0207793,MR0240167,MR0160807}. The analysis of these objects leads to a better understanding of the Lie algebra itself. Moreover, deformations of Lie algebras can yield families of new Lie algebras. \\
In the so-called \textit{continuous} cohomology, many results about the classical infinite-dimensional Lie algebras such as the Witt algebra are known. In fact, a geometrical realization of the Witt algebra is given by the complexified Lie algebra of polynomial vector fields on the circle, which forms a dense subalgebra of the complexified Lie algebra of smooth vector fields on the circle, $Vect(S^1)$. In this setting, it is natural to consider continuous cohomology. The continuous cohomology of vector fields on the circle with values in the trivial module is known, see the results by Gelfand and Fuks \cite{MR874337,MR0245035}. Similarly, based on results of Goncharova \cite{MR0402765}, Reshetnikov \cite{MR0292097} and Tsujishita \cite{MR0458517}, the vanishing of the continuous cohomology of vector fields on the circle with values in the adjoint module was proved by Fialowski and Schlichenmaier in \cite{MR2030563}.   \\
Less is known about the so-called \textit{algebraic} cohomology, also known as discrete cohomology, which contains the continuous cohomology as a sub-complex. The primary definition of the Witt and the Virasoro algebra is based on the Lie structure and is thus purely algebraic. In this article, we consider the Witt and the Virasoro algebra as purely algebraic objects, and we do not work in specific geometrical realizations of these. Hence, our results are independent of any underlying topology chosen. Similarly, our cochains are purely algebraic cochains, meaning we do not restrict ourselves to continuous cochains. Indeed, there are limitations for the continuous cohomology of purely algebraic infinite-dimensional Lie algebras, see e.g. Wagemann \cite{wagemann:tel-00397780}. Moreover, algebraic cohomology works for any base field $\mathbb{K}$ with characteristic zero, and not only for the fields $\mathbb{C}$ or $\mathbb{R}$.
Therefore, knowledge of the algebraic cohomology is needed.\\
In the literature, results on the algebraic cohomology of the Witt and the Virasoro algebra are somewhat scarce. In fact, algebraic cohomology is in general much harder to compute than continuous cohomology. In \cite{MR3200354,MR3363999}, Schlichenmaier showed the vanishing of the second algebraic cohomology of the Witt and the Virasoro algebra with values in the adjoint module by using elementary algebraic methods; see also Fialowski \cite{MR2985241}. For the Witt algebra, Fialowski announced this result already in \cite{MR1054321}, but without proof. Van den Hijligenberg and Kotchetkov proved in \cite{MR1417181} the vanishing of the second algebraic cohomology with values in the adjoint module of the superalgebras $k(1),k^+(1)$ and of their central extensions.
 In the case of the Witt algebra, the vanishing of the third algebraic cohomology with values in the adjoint module was proved by Ecker and Schlichenmaier \cite{Ecker:2017sen}. In \cite{Ecker:2017sen}, also the vanishing of the first algebraic cohomology of the Witt and the Virasoro algebra was shown. \\
The main aim of this article is to prove the one-dimensionality of the third algebraic cohomology of the Virasoro algebra with values in the adjoint module. This result was already announced in our previous article \cite{Ecker:2017sen}, but only a sketch of the proof was provided. The detailed proof is given in the present article. We use both elementary algebraic manipulations and higher tools from algebraic cohomology, such as long exact sequences and spectral sequences. Furthermore, we obtain several intermediate results, including the one-dimensionality of the third algebraic cohomology of the Witt and the Virasoro algebra with values in the trivial module. \\
 This article is organized as follows: In Section \ref{Section2}, we recall the algebraic definitions of the Witt and the Virasoro algebra. \\
In Section \ref{Section3}, we introduce the cohomology of Lie algebras, as well as some tools used to compute this cohomology. This section also contains a brief summary of the results known of the algebraic cohomology of the Witt and the Virasoro algebra, including the results derived in this article.\\
The final Section \ref{Section5} constitutes the main part of the present article and contains the proof of the one-dimensionality of the third algebraic cohomology of the Virasoro algebra with values in the adjoint module.\\
In a first step, using long exact sequences and spectral sequences, we show $\mathrm{H}^3(\mathcal{V},\mathcal{V})\cong \mathrm{H}^3(\mathcal{V},\mathbb{K})$.
In the proof, an earlier result of the authors  \cite{Ecker:2017sen}, $\mathrm{H}^3(\mathcal{W},\mathcal{W})=\{0\}$, was used.  \par
In a second step, we show that $\mathrm{H}^3(\mathcal{V},\mathbb{K})$ and simultaneously $\mathrm{H}^3(\mathcal{W},\mathbb{K})$ are one-dimensional.
Inspired by results from continuous cohomology, we identify a non-trivial 3-cocycle which we call \textit{algebraic Godbillon-Vey cocycle}. By essentially elementary but nevertheless intricate algebraic methods we show that this cocycle is a generator of $\mathrm{H}^3(\mathcal{V},\mathbb{K})$ and  $\mathrm{H}^3(\mathcal{W},\mathbb{K})$. Rather complicated recursions are needed to obtain this result. Hence indeed we obtain that $dim(\mathrm{H}^3(\mathcal{V},\mathcal{V}))=1$ whereas $\mathrm{H}^3(\mathcal{W},\mathcal{W})=\{0\}$. The reader should compare this with the results obtained for the second cohomology, where we have  $\mathrm{H}^2(\mathcal{W},\mathcal{W})=\mathrm{H}^2(\mathcal{V},\mathcal{V})=\{0\}$, see \cite{MR3200354}, \cite{MR2985241}.

\section{The Witt and the Virasoro Algebra}\label{Section2}
The \textit{Witt algebra} $\mathcal{W}$ is an infinite-dimensional, $\mathbb{Z}$-graded Lie algebra first introduced by Cartan in 1909 \cite{MR1509105}.
As a vector space, the Witt algebra is generated over a base field $\mathbb{K}$ with characteristic zero by the basis elements $\{ e_n\ |\ n\in\mathbb{Z}\}$, which satisfy the following Lie algebra structure equation:
\begin{equation*}
[e_n,e_m]=(m-n)e_{n+m},\qquad n,m\in\mathbb{Z}\,.
\end{equation*}
The Witt algebra is a $\mathbb{Z}$-graded Lie algebra, the degree of an element $e_n$ being defined by $deg(e_n):=n$. More precisely, the Witt algebra is an \textit{internally} $\mathbb{Z}$-graded Lie algebra, as the grading is given by one of its own elements, namely $e_0$: $[e_0,e_m]=me_{m}=deg(e_m)e_m$. The Witt algebra $\mathcal{W}$ can thus be decomposed into an infinite sum of one-dimensional homogeneous subspaces $\mathcal{W}_n$, where each subspace $\mathcal{W}_n$ is generated over $\mathbb{K}$ by a single element $e_n$.\par 
It is a well-known fact that the Witt algebra, up to equivalence and rescaling, has a unique non-trivial central extension $\mathcal{V}$, which in fact is a universal central extension:
\begin{equation}
0\longrightarrow \mathbb{K}\stackrel{i}{\longrightarrow}\mathcal{V}\stackrel{\pi}{\longrightarrow}\mathcal{W}\longrightarrow 0\,,\label{Extension}
\end{equation}
where $\mathbb{K}$ is in the center of $\mathcal{V}$. This extension $\mathcal{V}$ is called the Virasoro algebra.\\
As a vector space, $\mathcal{V}$ is given as a direct sum $\mathcal{V}=\mathbb{K}\oplus\mathcal{W}$, with generators $\hat{e}_n:=(0,e_n)$ and the one-dimensional central element $t:=(1,0)$. The generators fulfill the following Lie structure equation:
\begin{equation}\label{VirasoroLieStructure}
\begin{aligned}
&[\hat{e}_n,\hat{e}_m]=(m-n)\hat{e}_{n+m}+\alpha(e_n,e_m)\cdot t\qquad n,m\in\mathbb{Z}\,, \\
&[\hat{e}_n,t]=[t,t]=0\,,
\end{aligned}
\end{equation}
where $\alpha\in Z^2(\mathcal{W},\mathbb{K})$ is the so-called Virasoro 2-cocycle, sometimes also called the Gelfand-Fuks cocycle, which can be represented by:
\begin{equation}
\alpha(e_n,e_m)=-\frac{1}{12}(n^3-n)\delta_{n+m,0}\,.\label{GelfandFuks}
\end{equation}
 The cubic term $n^3$ is the most important term, while the linear term $n$ is  a coboundary\footnote{The symbol $\delta_{i,j}$ is the Kronecker Delta, defined as being one if $i=j$ and zero otherwise.}. \\
By defining $deg(\hat{e}_n):=deg(e_n)=n$ and $deg(t):=0$, the Virasoro algebra becomes also an internally $\mathbb{Z}$-graded Lie algebra.

\section{The cohomology of Lie algebras}\label{Section3}
\subsection{The Chevalley-Eilenberg cohomology}
For the convenience of the reader, we will briefly recall the Chevalley-Eilenberg cohomology, i.e. the cohomology of Lie algebras.\\
Let $\mathcal{L}$ be a Lie algebra and $M$ an $\mathcal{L}$-module. We denote by $C^q(\mathcal{L},M)$ the space of $q$-multilinear alternating maps on $\mathcal{L}$ with values in $M$, 
\begin{equation*}
C^q(\mathcal{L},M):=\text{Hom}_\mathbb{K}(\wedge^q\mathcal{L},M)\,.
\end{equation*}
Elements of $C^q(\mathcal{L},M)$ are called $q$-cochains. By convention, we have $C^0(\mathcal{L},M):=M$. The coboundary operators $\delta_q$ are defined by:
\begin{align}
\begin{array}[h]{rl}
\forall q\in\mathbb{N},\qquad &\delta_q:C^q(\mathcal{L},M)\rightarrow C^{q+1}(\mathcal{L},M): \psi \mapsto \delta_q \psi \,,\\
&\\
 (\delta_q\psi)(x_1,\dots x_{q+1}):&=\sum_{1\leq i<j\leq q+1}(-1)^{i+j+1}\ \psi (\left[x_i,x_j\right],x_1,\dots , \hat{x}_i,\dots , \hat{x}_j,\dots ,x_{q+1})\\
&\\
& +\sum_{i=1}^{q+1}(-1)^i\ x_i\cdot \psi (x_1,\dots ,\hat{x}_i,\dots ,x_{q+1})\,,
\end{array}\label{Delta}
\end{align}
with $x_1,\dots , x_{q+1}\in\mathcal{L}$, $\hat{x}_i$ means that the entry $x_i$ is omitted and the dot $\cdot$ stands for the module structure. For $x\in\mathcal{L}$ and $y\in M$ we have $x\cdot y=[x,y]$ in case of the adjoint module $M=\mathcal{L}$, and $x\cdot y=0$ in case of the trivial module $M=\mathbb{K}$. The coboundary operators satisfy $\delta_{q+1}\circ \delta_q=0\ \forall\ q\in\mathbb{N}$, meaning we obtain a cochain complex $(C^*(\mathcal{L},M),\delta )$ called the \textit{Chevalley-Eilenberg complex}. The corresponding cohomology is the \textit{Chevalley-Eilenberg cohomology} defined by: 
\begin{equation*} 
\mathrm{H}^q(\mathcal{L},M):=Z^q(\mathcal{L},M)/B^q(\mathcal{L},M)\,,
\end{equation*}
where $Z^q(\mathcal{L},M):=\text{ ker }\delta_q$ is the vector space of $q$-cocycles and 
$B^q(\mathcal{L},M):=\text{ im }\delta_{q-1}$ is the vector space of $q$-coboundaries. For more details, we refer the reader to the original literature by Chevalley and Eilenberg \cite{MR0024908}.
\subsection{Degree of a homogeneous cochain}
 Let $\mathcal{L}$ be a $\mathbb{Z}$-graded Lie algebra $\mathcal{L}=\bigoplus_{n\in\mathbb{Z}}\mathcal{L}_n$ and $M$ a $\mathbb{Z}$-graded $\mathcal{L}$-module, i.e. $M=\bigoplus_{n\in\mathbb{Z}}M_n$. A $q$-cochain $\psi$ is \textit{homogeneous of degree d} if there exists a $d\in\mathbb{Z}$ such that for all $q$-tuple $x_1,\dots ,x_q$ of homogeneous elements $x_i\in\mathcal{L}_{deg(x_i)}$, we have:
\begin{equation*}
\psi (x_1,\dots ,x_q)\in M_n\text{  with  }n=\sum_{i=1}^q deg (x_i)+d\,.
\end{equation*}
This leads to the decomposition of the cohomology for all $q$:
\begin{equation*}
\mathrm{H^q}(\mathcal{L},M)=\bigoplus_{d\in\mathbb{Z}}\mathrm{H}^q_{(d)}(\mathcal{L},M)\,.
\end{equation*} 
An important result by Fuks \cite{MR874337} states that for internally graded Lie algebras and modules, the cohomology reduces to the degree-zero cohomology:
\begin{equation}\label{Fuks}
\begin{aligned}
&\mathrm{H}^q_{(d)}(\mathcal{L},M)=\{0\}\qquad\text{ for }d\neq 0\,, \\
&\mathrm{H}^q(\mathcal{L},M)=\mathrm{H}^q_{(0)}(\mathcal{L},M)\,.
\end{aligned}
\end{equation}

\subsection{Results on the algebraic cohomology of the Witt and the Virasoro algebra}\label{ResultsWittVirasoro}
For future reference, we briefly summarize in this section known results on the algebraic cohomology of the Witt and the Virasoro algebra, including the results derived in the present article. Concerning the interpretation of the low-dimensional cohomology, we refer the reader to \cite{Ecker:2017sen}.\\
For the zeroth cohomology corresponding to invariants, we immediately obtain by direct computation the following results for the Witt and the Virasoro algebra:
\begin{align*}
\mathrm{H}^0(\mathcal{W},\mathbb{K})=\mathbb{K}&\qquad\text{ and }\qquad\mathrm{H}^0(\mathcal{W},\mathcal{W})=\{0\}\,, \\
\mathrm{H}^0(\mathcal{V},\mathbb{K})=\mathbb{K}&\qquad\text{ and }\qquad\mathrm{H}^0(\mathcal{V},\mathcal{V})=\mathbb{K}\ t\,, 
\end{align*} 
where $t$ is the central element. \\
In \cite{Ecker:2017sen}, the first algebraic cohomology of the Witt and the Virasoro algebra was computed:
\begin{align*}
\mathrm{H}^1(\mathcal{W},\mathbb{K})=\{0\}&\qquad\text{ and }\qquad \mathrm{H}^1(\mathcal{W},\mathcal{W})=\{0\}\,,\\
\mathrm{H}^1(\mathcal{V},\mathbb{K})=\{0\} &\qquad\text{ and }\qquad\mathrm{H}^1(\mathcal{V},\mathcal{V})=\{0\}\,. 
\end{align*}
Concerning the second cohomology related to central extensions and deformations, we have the following results:
\begin{align*}
dim(\mathrm{H}^2(\mathcal{W},\mathbb{K}))=1&\qquad\text{ and }\qquad \mathrm{H}^2(\mathcal{W},\mathcal{W})=\{0\}\,,\\
\mathrm{H}^2(\mathcal{V},\mathbb{K})=\{0\} &\qquad\text{ and }\qquad\mathrm{H}^2(\mathcal{V},\mathcal{V})=\{0\}\,.
\end{align*}
The first result $dim(\mathrm{H}^2(\mathcal{W},\mathbb{K}))=1$ is a well-known result. It states that the Witt algebra admits, up to equivalence and rescaling, only one non-trivial central extension, namely the Virasoro algebra. 
For an algebraic proof of this result, see e.g. \cite{MR3185361}. The second result $\mathrm{H}^2(\mathcal{W},\mathcal{W})=\{0\}$ was announced by Fialowski \cite{MR1054321} without proof and was shown algebraically by Schlichenmaier \cite{MR3200354,MR3363999} and Fialowski \cite{MR2985241}. This result implies that the Witt algebra is infinitesimally and formally rigid. 
The third result $\mathrm{H}^2(\mathcal{V},\mathbb{K})=\{0\} $ and the fourth result $\mathrm{H}^2(\mathcal{V},\mathcal{V})=\{0\}$ were shown by Schlichenmaier \cite{MR3200354}. \\
Concerning the third cohomology related to crossed modules, we have the following results:
\begin{align*}
dim (\mathrm{H}^3(\mathcal{W},\mathbb{K}))=1&\qquad\text{ and }\qquad \mathrm{H}^3(\mathcal{W},\mathcal{W})=\{0\}\,,\\
dim(\mathrm{H}^3(\mathcal{V},\mathbb{K}))=1& \qquad\text{ and }\qquad dim(\mathrm{H}^3(\mathcal{V},\mathcal{V}))=1\,. 
\end{align*}
The second result $\mathrm{H}^3(\mathcal{W},\mathcal{W})=\{0\}$ was algebraically proved by Ecker and Schlichenmaier \cite{Ecker:2017sen}. This result states that there are no crossed modules associated to the Lie algebra $\mathcal{W}$ and the module $\mathcal{W}$.
The remaining three results were announced in \cite{Ecker:2017sen}, and a sketch of the proof was provided. The details of the proof are given in the present article, see Theorems \ref{H3VKMain} and \ref{h3vv}. In these cases there exists an equivalence class of a crossed module associated to $\mathcal{W}$ and $\mathbb{K}$, $\mathcal{V}$ and $\mathbb{K}$, as well as $\mathcal{V}$ and $\mathcal{V}$, respectively. Note that for $\mathrm{H}^3(\mathcal{W},\mathbb{K})$ and $\mathrm{H}^3(\mathcal{V},\mathbb{K})$, we will provide below an explicit algebraic expression for the cocycle generating these spaces.

\subsection{The Hochschild-Serre Spectral Sequence}\label{Section4}

In the present article, we will use the Hochschild-Serre spectral sequence. The following theorem is a well-known result in algebraic cohomology:
  \newtheorem{HS}{Theorem [Hochschild-Serre]}[subsection]
\begin{HS} \label{hs}
 For every ideal $\mathfrak{h}$ of a Lie algebra $\mathfrak{g}$, there is a convergent first quadrant spectral sequence:
\begin{equation*}
E_2^{pq}=\mathrm{H}^p(\mathfrak{g}/\mathfrak{h},\mathrm{H}^q(\mathfrak{h},M)))\Rightarrow \mathrm{H}^{p+q}(\mathfrak{g},M)\,,
\end{equation*}
with $M$ being a $\mathfrak{g}$-module and via $\mathfrak{h}\hookrightarrow \mathfrak{g}$ also a $\mathfrak{h}$-module.
\end{HS}
A proof of this well-known result can for example be found in the textbook of Weibel \cite{MR1269324}. The original literature is given by the articles \cite{MR0052438,MR0054581} by Hochschild and Serre.\\
For knowledge of general spectral sequences, the reader may consult the textbook by McCleary \cite{MR1793722}.

\section{Analysis of \texorpdfstring{$\mathrm{H}^3(\mathcal{V},\mathcal{V})$}{H3(V,V)}}\label{Section5}
The aim of this article is to prove the following result: 
\newtheorem{H3VV}{Theorem}[section]
\begin{H3VV} \label{h3vv}
The third cohomology group of the Virasoro algebra $\mathcal{V}$ over a field $\mathbb{K}$ with $char(\mathbb{K})=0$ and values in the adjoint module is one-dimensional, i.e.
\begin{equation*}
dim(\mathrm{H}^3(\mathcal{V},\mathcal{V}))=1\,.
\end{equation*}
\end{H3VV}
\begin{proof}
Theorem \ref{h3vv} can be shown by direct computations using elementary algebra. However, the resulting proof is rather long and complicated. We prefer to prove this result by using higher tools from cohomological algebra, which has also the advantage to yield interesting intermediate results. Moreover, it exhibits the relations between the various cohomology groups in an efficient way.\\ 
The proof is divided into two main steps:
\begin{itemize}
	\item Theorem \ref{H3VW}: Proof of $\mathrm{H}^3(\mathcal{V},\mathcal{W})\cong \mathrm{H}^3(\mathcal{W},\mathcal{W})$.
	\item Theorem \ref{H3VKMain}: Proof of $dim(\mathrm{H}^3(\mathcal{V},\mathbb{K}))=dim(\mathrm{H}^3(\mathcal{W},\mathbb{K}))= 1$.
\end{itemize}
Theorem \ref{H3VW} will be derived via spectral sequences.  Theorem \ref{H3VKMain} will be shown by using elementary algebraic manipulations. 
 The result for the Virasoro algebra listed in the second  bullet point is necessary for the proof of $dim(\mathrm{H}^3(\mathcal{V},\mathcal{V}))=1$. \\

The short exact sequence of Lie algebras in (\ref{Extension}) can also be viewed as a short exact sequence of $\mathcal{V}$-modules. In fact, the Witt algebra $\mathcal{W}$ is a $\mathcal{V}$-module.
Furthermore, the base field $\mathbb{K}$ is a trivial $\mathcal{V}$-module.
This short exact sequence of $\mathcal{V}$-modules gives rise to a long exact sequence in cohomology, the relevant part for us being the third cohomology:\\
\begin{equation}
\dots\rightarrow  \mathrm{H}^2(\mathcal{V},\mathcal{W})\rightarrow \mathrm{H}^3(\mathcal{V},\mathbb{K})\rightarrow \mathrm{H}^3(\mathcal{V},\mathcal{V})\rightarrow  \mathrm{H}^3(\mathcal{V},\mathcal{W})\rightarrow \dots\,.\label{LongSequence}
\end{equation}
 Concerning the third cohomology group, we have $\mathrm{H}^3(\mathcal{W},\mathcal{W})=\{0\}$ \cite{Ecker:2017sen}, see Section \ref{ResultsWittVirasoro}. This result, together with Theorem \ref{H3VW}, leads to $\mathrm{H}^3(\mathcal{V},\mathcal{W})=\{0\}$. Moreover, in \cite{MR3200354}, it was shown that $\mathrm{H}^2(\mathcal{V},\mathcal{W})\cong \mathrm{H}^2(\mathcal{W},\mathcal{W})$ and also $ \mathrm{H}^2(\mathcal{W},\mathcal{W})=\{0\}$; both results together leading to $\mathrm{H}^2(\mathcal{V},\mathcal{W})=0$.
Consequently, the long exact sequence (\ref{LongSequence}) above reduces for the third cohomology to:
\begin{equation}
0\rightarrow \mathrm{H}^3(\mathcal{V},\mathbb{K})\rightarrow \mathrm{H}^3(\mathcal{V},\mathcal{V})\rightarrow  0\,.\label{TwoTermSequence}
\end{equation}
As $dim(\mathrm{H}^3(\mathcal{V},\mathbb{K}))= 1$ by Theorem \ref{H3VKMain}, the two-term exact sequence (\ref {TwoTermSequence}) above yields the desired result $dim(\mathrm{H}^3(\mathcal{V},\mathcal{V}))=1$. 
\end{proof}

\subsection{Proof of \texorpdfstring{$\mathrm{H}^3(\mathcal{V},\mathcal{W})\cong \mathrm{H}^3(\mathcal{W},\mathcal{W})$}{H3(V,W)=H3(W,W)}}\label{SectionMain1}
We will prove this via the Hochschild-Serre spectral sequence\footnote{There is another proof available by elementary but very tedious calculations. The proof using spectral sequences is much shorter.}. \\  
Take $\mathfrak{g}=\mathcal{V}$ and $\mathfrak{h}=\mathbb{K}$ in Theorem \ref{hs}, then $\mathfrak{g}/\mathfrak{h}=\mathcal{W}$. Hence, the second stage spectral sequence $E_2^{pq}$ becomes in our case:
\begin{equation*}
E_2^{p,q}=\mathrm{H}^p(\mathcal{W},\mathrm{H}^q(\mathbb{K},M))\Rightarrow \mathrm{H}^{p+q}(\mathcal{V},M)\,.
\end{equation*}
Due to the alternating property of the cochains, we have $\mathrm{H}^k(\mathbb{K},M)=0$ for $k>1$. This is true for any module $M$. \\
Consequently, the Hochschild-Serre spectral sequence has only two lines in our case:
\begin{equation}\label{HochSerre}
\def\arraystretch{1.6}
\begin{array}[h]{cccc}
 0 & 0& 0 & \dots\\
 \mathrm{H}^0(\mathcal{W},\mathrm{H}^1(\mathbb{K},M)) & \mathrm{H}^1(\mathcal{W},\mathrm{H}^1(\mathbb{K},M)) & \mathrm{H}^2(\mathcal{W},\mathrm{H}^1(\mathbb{K},M)) & \dots\\
  \framebox{$\mathrm{H}^0(\mathcal{W},\mathrm{H}^0(\mathbb{K},M))$} & \mathrm{H}^1(\mathcal{W},\mathrm{H}^0(\mathbb{K},M)) & \mathrm{H}^2(\mathcal{W},\mathrm{H}^0(\mathbb{K},M)) & \dots
\end{array}
\end{equation}
For reference, we put the $(0,0)$-entry into a box. The entries with $p<0$ or $q<0$ are zero.
In addition, the entries for $q\geq 2$ are zero because of $\mathrm{H}^k(\mathbb{K},M)=0$ for $k>1$.\\
The second stage spectral sequence comes with differentials:
\begin{equation*}
d_2^{p,q}: E_2^{p,q}\longrightarrow E_2^{p+2,q-1}\,.
\end{equation*}
 With these maps, we can take the cohomology $E_3^{p,q}$ of $E_2^{p,q}$, which gives the third page spectral sequence:
\begin{equation*}
E_3^{p,q}=\frac{\text{ker } d_2^{p,q}}{\text{im } d_2^{p-2,q+1}}\,.
\end{equation*}
The third page spectral sequence has the same shape as the second page spectral sequence, meaning it too has only two lines different from zero.
Again, the third page spectral sequence comes with maps:
\begin{equation*}
d_3^{p,q}: E_3^{p,q}\longrightarrow E_3^{p+3,q-2}\,.
\end{equation*}
However, the operator $d_3^{p,q}$ corresponds to going three entries to the right and two entries to the bottom. Since we only have two lines different from zero , we always obtain $\text{im } d_3^{p,q}=0$. Moreover, the kernel of $d_3^{p,q}$ then corresponds to $E_3^{p,q}$. Therefore, we obtain $E_3^{p,q}=E_4^{p,q}=\dots =E_\infty^{p,q}$.
We finally gathered all the necessary ingredients to prove the following result:
\newtheorem{H3VW}{Theorem}[subsection]
\begin{H3VW} \label{H3VW}
The following holds:
\begin{align*}
\text{If  }\qquad&\mathrm{H}^j(\mathcal{W},\mathcal{W})=0\quad\text{for }\quad k-2\leq j\leq k-1\,,\\
\text{ then }\qquad&\mathrm{H}^k(\mathcal{V},\mathcal{W})\cong\mathrm{H}^k(\mathcal{W},\mathcal{W})\,.\\
\text{In particular,}\qquad 
&\mathrm{H}^3(\mathcal{V},\mathcal{W})\cong \mathrm{H}^3(\mathcal{W},\mathcal{W})\,.
\end{align*}
\end{H3VW}
\begin{proof}
If $M$ is a $\mathfrak{g}$-module, then $M$ is also a $\mathfrak{h}$-module since $\mathfrak{h}\unlhd\mathfrak{g}$ in Theorem \ref{hs}. Then $\mathfrak{g}/\mathfrak{h}$ acts on $\mathrm{H}^q(\mathfrak{h},M)$.
We consider the module $M=\mathcal{W}$ in (\ref{HochSerre}), which is a $\mathcal{V}$-module and a trivial $\mathbb{K}$-module. The latter implies $\mathrm{H}^0(\mathbb{K},\mathcal{W})=^{\mathbb{K}}\mathcal{W}=\mathcal{W}$, where $^{\mathbb{K}}\mathcal{W}$ denotes the space of $\mathbb{K}$-invariants of $\mathcal{W}$.

 Moreover, we have $\mathrm{H}^1(\mathbb{K},\mathcal{W})=\mathcal{W}$ as every linear map $\phi\in C^1(\mathbb{K},\mathcal{W})$ is a cocycle and hence $\mathrm{H}^1(\mathbb{K},\mathcal{W})$ corresponds one-to-one to all linear maps $1\mapsto \omega\in\mathcal{W}$, i.e. to all elements of $\mathcal{W}$. The $(0,0)$ and $(0,1)$ entries become: $\mathrm{H}^0(\mathcal{W},\mathcal{W})=^{\mathcal{W}}\mathcal{W}=0$.

Our second stage spectral sequence thus becomes:
\begin{displaymath}
\xymatrix{
  0  & 0& 0 &0& 0 & \dots\\
 0 \ar[drr]^{d_2^{0,1}} & \mathrm{H}^1(\mathcal{W},\mathcal{W})\ar[drr]^{d_2^{1,1}} & \mathrm{H}^2(\mathcal{W},\mathcal{W})\ar[drr]^{d_2^{2,1}}& \mathrm{H}^3(\mathcal{W},\mathcal{W}) & \mathrm{H}^4(\mathcal{W},\mathcal{W})& \dots\\
  \framebox{$0$} & \mathrm{H}^1(\mathcal{W},\mathcal{W}) & \mathrm{H}^2(\mathcal{W},\mathcal{W})& \mathrm{H}^3(\mathcal{W},\mathcal{W}) & \mathrm{H}^4(\mathcal{W},\mathcal{W}) & \dots\\
}
\end{displaymath}
Once again, we put the $(0,0)$-entry into a box to increase readability.
In order to simplify the notation, we define: 
\begin{equation}
\varphi_p:=d_2^{p,1}: E_2^{p,1}\longrightarrow E_2^{p+2,0}\,.\label{DiffMap}
\end{equation}
Next, we take the cohomology of the sequence with respect to $\varphi_p$, which gives us the third page spectral sequence $E_3^{p,q}=E^{p,q}_\infty$. We will abbreviate $\mathrm{H}^i(\mathcal{W},\mathcal{W})$ by $\mathrm{H}^i$:
\vspace{-0.5cm}
\begin{center}
\resizebox{\textwidth}{!}{\xymatrix{
  0 \ar@{-}[dr] & 0\ar@{-}[dr]& 0\ar@{-}[dr] &0\ar@{-}[dr]& 0 & \dots\\
 0 \ar@{-}[dr] & \text{ker }(\varphi_1:\mathrm{H}^1\rightarrow \mathrm{H}^3)\ar@{-}[dr]& \text{ker }(\varphi_2:\mathrm{H}^2\rightarrow \mathrm{H}^4)\ar@{-}[dr]& \text{ker }(\varphi_3:\mathrm{H}^3\rightarrow \mathrm{H}^5)\ar@{-}[dr] &\text{ker }(\varphi_4:\mathrm{H}^4\rightarrow \mathrm{H}^6)& \dots\\
  \framebox{$0$} & \mathrm{H}^1 \ar@{-->}[d] & \mathrm{H}^2  \ar@{-->}[d]& \frac{\mathrm{H}^3}{\text{im }\varphi_1} \ar@{-->}[d] & \frac{\mathrm{H}^4}{\text{im }\varphi_2} \ar@{-->}[d]& \dots\\
   & \mathrm{H}^1(\mathcal{V},\mathcal{W})&\mathrm{H}^2(\mathcal{V},\mathcal{W}) &\mathrm{H}^3(\mathcal{V},\mathcal{W})&\mathrm{H}^4(\mathcal{V},\mathcal{W}) & \dots 
}
}\end{center} \null 

The elements $E^{p,q}_2$ converge to $\mathrm{H}^{p+q}(\mathfrak{g},M)$, i.e. $\mathrm{H}^{p+q}(\mathcal{V},\mathcal{W})$ in our case. This means that in the case under consideration, we can write the elements $\mathrm{H}^{p+q}(\mathcal{V},\mathcal{W})$ of degree $n=p+q$ as a direct sum of the elements $E^{p,q}_3$ of degree $n$, i.e. the elements $E^{p,q}_3$ with $p+q=n$ lying on the $n$-th diagonal: 
\begin{equation*}
\mathrm{H}^{k}(\mathcal{V},\mathcal{W})\cong E_3^{k,0}\oplus E_3^{k-1,1}\,.
\end{equation*}
As illustrated in the diagram above, we thus explicitly obtain the following results:
\begin{itemize}
	\item $\mathrm{H}^1(\mathcal{V},\mathcal{W})\cong \mathrm{H}^1(\mathcal{W},\mathcal{W})$.
	\item[]
	\item $\mathrm{H}^2(\mathcal{V},\mathcal{W})\cong\mathrm{H}^2(\mathcal{W},\mathcal{W})\oplus\text{ker }(\varphi_1:\mathrm{H}^1(\mathcal{W},\mathcal{W})\rightarrow \mathrm{H}^3(\mathcal{W},\mathcal{W}))$.
	\item[]
	\item $\mathrm{H}^3(\mathcal{V},\mathcal{W})\cong\frac{\mathrm{H}^3(\mathcal{W},\mathcal{W})}{\text{im }\varphi_1}\oplus\text{ker }(\varphi_2:\mathrm{H}^2(\mathcal{W},\mathcal{W})\rightarrow \mathrm{H}^4(\mathcal{W},\mathcal{W}))$.
	\item[]
	\item $\mathrm{H}^4(\mathcal{V},\mathcal{W})\cong\frac{\mathrm{H}^4(\mathcal{W},\mathcal{W})}{\text{im }\varphi_2}\oplus\text{ker }(\varphi_3:\mathrm{H}^3(\mathcal{W},\mathcal{W})\rightarrow \mathrm{H}^5(\mathcal{W},\mathcal{W}))$.
	\item[]
	\item \dots
\end{itemize}
The first result has already been computed explicitly in \cite{Ecker:2017sen}. Regarding the second result, using the fact from Section \ref{ResultsWittVirasoro} that $\mathrm{H}^1(\mathcal{W},\mathcal{W})=\{0\}$, we obtain $\text{ker }(\varphi_1:\mathrm{H}^1(\mathcal{W},\mathcal{W})\rightarrow \mathrm{H}^3(\mathcal{W},\mathcal{W}))=0$ and thus $\mathrm{H}^2(\mathcal{V},\mathcal{W})=\mathrm{H}^2(\mathcal{W},\mathcal{W})$. This has also been proven by explicit computations before in \cite{MR3200354}.
Concerning the third result, we have $\text{ker }(\varphi_2:\mathrm{H}^2(\mathcal{W},\mathcal{W})\rightarrow \mathrm{H}^4(\mathcal{W},\mathcal{W}))=0$ because $\mathrm{H}^2(\mathcal{W},\mathcal{W})=\{0\}$, see Section \ref{ResultsWittVirasoro}. Moreover, we have $\text{im }(\varphi_1:\mathrm{H}^1(\mathcal{W},\mathcal{W})\rightarrow \mathrm{H}^3(\mathcal{W},\mathcal{W}))=0$ because $\mathrm{H}^1(\mathcal{W},\mathcal{W})=\{0\}$. Hence, we obtain the following result:
\begin{equation*}
\mathrm{H}^3(\mathcal{V},\mathcal{W})\cong\mathrm{H}^3(\mathcal{W},\mathcal{W})\,.
\end{equation*} 
This can be generalized. We see from the diagram above that for general $k\geq 3$, we have:
\begin{equation*}
\mathrm{H}^k(\mathcal{V},\mathcal{W})\cong\frac{\mathrm{H}^k(\mathcal{W},\mathcal{W})}{\text{im }\varphi_{k-2}}\oplus\text{ker }(\varphi_{k-1}:\mathrm{H}^{k-1}(\mathcal{W},\mathcal{W})\rightarrow \mathrm{H}^{k+1}(\mathcal{W},\mathcal{W})).
\end{equation*}
The announced general result is thus given by:
\begin{equation*}
\text{If }\quad \mathrm{H}^j(\mathcal{W},\mathcal{W})=0\quad\text{for}\quad k-2\leq j\leq k-1\,,
\text{ then}\quad  \mathrm{H}^k(\mathcal{V},\mathcal{W})\cong\mathrm{H}^k(\mathcal{W},\mathcal{W})\,.
\end{equation*}
\end{proof}
\subsection{Analysis of \texorpdfstring{$\mathrm{H}^3(\mathcal{W},\mathbb{K})$}{H3(W,K)} and \texorpdfstring{$\mathrm{H}^3(\mathcal{V},\mathbb{K})$}{H3(V,K)}}\label{SectionMain2}
In this section we will prove:
\newtheorem{H3VKMain}{Theorem}[subsection]
\begin{H3VKMain} \label{H3VKMain}
The third cohomology group of the Witt and the Virasoro algebra with values in the trivial module $\mathbb{K}$ is one-dimensional, i.e.:
\begin{equation*}
dim(\mathrm{H}^3(\mathcal{W},\mathbb{K}))=dim(\mathrm{H}^3(\mathcal{V},\mathbb{K}))= 1\,.
\end{equation*}
\end{H3VKMain}
For the convenience of the reader, we write down the condition for a 3-cochain $\psi$ to be a cocycle with values in the trivial module:
\begin{equation}\label{CocGen}
\begin{aligned}
&(\delta_3\psi)(x_1,x_2,x_3,x_4)=\psi\left(\left[x_1,x_2\right],x_3,x_4\right)-\psi\left(\left[x_1,x_3\right],x_2,x_4\right)+\psi\left(\left[x_1,x_4\right],x_2,x_3\right)\\
&+\psi\left(\left[x_2,x_3\right],x_1,x_4\right)-\psi\left(\left[x_2,x_4\right],x_1,x_3\right)
+\psi\left(\left[x_3,x_4\right],x_1,x_2\right)=0\,,
\end{aligned}
\end{equation}
where $x_1,x_2,x_3,x_4$ are elements of $\mathcal{W}$ or $\mathcal{V}$.\\
The condition for a 3-cocycle $\psi$ to be a coboundary with values in the trivial module is given by:
\begin{equation}
\psi(x_1,x_2,x_3)=(\delta_2\phi)(x_1,x_2,x_3)=\phi\left(\left[x_1,x_2\right],x_3\right)+\phi\left(\left[x_2,x_3\right],x_1\right)+\phi\left(\left[x_3,x_1\right],x_2\right)\,,\label{CobGen}
 \end{equation}
where $\phi$ is a 2-cochain with values in $\mathbb{K}$ and $x_1,x_2,x_3$ are elements of  $\mathcal{W}$ or $\mathcal{V}$.\\
Next, consider the following trilinear map:
\begin{equation*}
\Psi :\mathcal{W}\times\mathcal{W}\times\mathcal{W}\rightarrow\mathbb{K}\,,
\end{equation*}
defined on the basis elements via:
\begin{equation}
\Psi (e_i,e_j,e_k)=(i-j)(j-k)(i-k)\delta_{i+j+k,0}\,,\label{GVCoc}
\end{equation}
which we extend trivially to:
\begin{equation*}
\hat{\Psi}:\mathcal{V}\times\mathcal{V}\times\mathcal{V}\rightarrow\mathbb{K}\,,
\end{equation*}
by setting $\hat{\Psi}(x_1,x_2,x_3)=0$ whenever one of the elements $x_1$, $x_2$ or $x_3$ is a multiple of the central element $t$.
\newtheorem{GV}{Proposition}[subsection]
\begin{GV} \label{gv}
The trilinear maps $\Psi$ and $\hat{\Psi}$ define non-trivial cocycle classes of $\mathrm{H}^3(\mathcal{W},\mathbb{K})$ and  $\mathrm{H}^3(\mathcal{V},\mathbb{K})$, respectively.
\end{GV}
\begin{proof}
By their very definition, $\Psi$ and $\hat{\Psi}$ are alternating.\\
A straight-forward calculation of (\ref{CocGen}) for a triplet of basis elements $e_i,e_j,e_k$ yields $\delta_3\Psi =0$. Hence $\Psi$ is a three-cocycle of $\mathcal{W}$.  
Concerning the Virasoro algebra, we have $\delta_3\hat{\Psi} (x_1,x_2,x_3,x_4)=0$ if one of the arguments is central. If all the arguments are coming from $\mathcal{W}$, we obtain\footnote{In abuse of notation, we use the same symbol $x$ to refer both to $x\in\mathcal{V}$ and its projection $\pi (x)\in\mathcal{W}$.}: $\delta_3\hat{\Psi} (x_1,x_2,x_3,x_4)=\delta_3\Psi (x_1,x_2,x_3,x_4)=0$.  Thus, $\hat{\Psi}$ is a 3-cocycle for $\mathcal{V}$. It remains to be shown that these cocycles are not trivial.\\
Let us assume that $\Psi$ and $\hat{\Psi}$ are coboundaries, which will lead us to a contradiction. So, let $\Phi:\mathcal{W}\times\mathcal{W}\rightarrow \mathbb{K}$ be a 2-cochain with $\Psi=\delta_2\Phi$. On the one hand, evaluating $\Psi$ on the triple $e_{-1},e_1,e_0$, we obtain by the very definition (\ref{GVCoc}) of $\Psi$:
\begin{equation}
 \Psi (e_{-1},e_1,e_0)=2\,.\label{gun}
\end{equation}
On the other hand, $\Psi$ being a coboundary we obtain using (\ref{CobGen}):
\begin{equation}\label{enb}
\begin{aligned}
\Psi (e_{-1},e_1,e_0)&=\Phi\left(\left[e_{-1},e_{1}\right],e_{0}\right)+\Phi\left(\left[e_{1},e_{0}\right],e_{-1}\right)+\Phi\left(\left[e_{0},e_{-1}\right],e_{1}\right)\\
&=2\Phi (e_0,e_0)-\Phi (e_1,e_{-1})-\Phi (e_{-1},e_1)=0\,.
\end{aligned}
\end{equation}
Therefore, $\Psi$ cannot be a coboundary. Similarly, assume $\hat{\Psi}=\delta_2\hat{\Phi}$, with $\hat{\Phi}:\mathcal{V}\times\mathcal{V}\rightarrow \mathbb{K}$. Again, we obtain $\hat{\Psi} (e_{-1},e_1,e_0)=\Psi (e_{-1},e_1,e_0)=2$ as well as $\delta_2\hat{\Phi}(e_{-1},e_1,e_0)=0$, hence $\hat{\Psi}$ cannot be a coboundary. Note that for pairs of elements of $e_{-1},e_1,e_0$, the defining cocycle (\ref{GelfandFuks}) for the central extension of $\mathcal{W}$ vanishes, hence exactly the same expression (\ref{enb}) will also appear for $\hat{\Psi}$.
\end{proof}
We call the cocycle $\Psi$ the \textit{algebraic Godbillon-Vey cocycle}. An immediate consequence of the  Proposition \ref{gv} is:
\newtheorem{H3VK}[GV]{Proposition}
\begin{H3VK} \label{H3VK}
The third cohomology group of the Witt and the Virasoro algebra with values in the trivial module is at least one-dimensional, i.e.:
\begin{equation*}
dim(\mathrm{H}^3(\mathcal{W},\mathbb{K}))\geq 1 \text{ and } dim(\mathrm{H}^3(\mathcal{V},\mathbb{K}))\geq 1\,.
\end{equation*}
\end{H3VK}
\theoremstyle{remark}
\newtheorem{R1}{Remark}[subsection]
\begin{R1}
The Godbillon-Vey cocycle is known in the context of the continuous cohomology of $\mathrm{H}^3(Vect (S^1),\mathbb{R})$. For the interested reader, we exhibit the relation in the following.\\
Let $t$ be the coordinate along $S^1$. The elements of $Vect(S^1)$ can be represented by functions on $S^1$. Assigning to the vector field $f(t)\frac{d}{dt}$ the function $f(t)$, it was shown in \cite{MR0245035} that the continuous cohomology $\mathrm{H}_c^*(Vect(S^1),\mathbb{R})$ of $Vect(S^1)$ with values in $\mathbb{R}$ is the free graded-commutative algebra
generated by an element $\omega$ of degree two and an element $\mathscr{GV}$ of degree three.
The generator of degree two is given by:
\begin{equation}\label{GenDeg2}
\omega: \left(f\frac{d}{dt},g\frac{d}{dt}\right)\mapsto \int_{S^1}\text{det}\left(
\begin{array}[h]{cc}
	f' & g'\\
	f''&g''
\end{array}\right)dt\,,
\end{equation} 
while the generator of degree three is given by:
\begin{equation}\label{GenDeg3}
\mathscr{GV}:\left(f\frac{d}{dt},g\frac{d}{dt},h\frac{d}{dt}\right)\mapsto \int_{S^1}\text{det}\left(
\begin{array}[h]{ccc}
	f & g & h\\
	f' & g'& h'\\
	f''&g''& h''
\end{array}\right)dt\,,
\end{equation} 
with $f,g,h\in C^\infty(S^1)$ and the prime denoting the derivative with respect to $t$. The generator $\mathscr{GV}$ in (\ref{GenDeg3}) is commonly called the \textit{Godbillon-Vey} cocycle. \\
If one considers the complexified vector field $\tilde{e}_n=ie^{int}\frac{d}{dt}$ then
\begin{equation}\label{egus}
\begin{aligned}
\mathscr{GV}(\tilde{e}_n,\tilde{e}_m,\tilde{e}_k)&=-\int_{S^1}\text{det }\left( \begin{array}[h]{ccc} 1 & 1 & 1\\ n & m& k\\ n^2 & m^2 & k^2\end{array}\right)\ e^{i(n+m+k)t}dt\\
&=(n-m)(n-k)(m-k)\int_{S^1}e^{i(n+m+k)t}dt
\end{aligned}
\end{equation}
The integral evaluates to zero if $n+m+k\neq 0$, otherwise it yields the value $1$. The expression (\ref{egus}) makes perfect sense for our algebraic generators $e_n$ of $\mathcal{W}$ for every field $\mathbb{K}$ with $\text{char }(\mathbb{K})=0$, and we obtain the expression (\ref{GVCoc}).
\end{R1}
Before continuing, we decompose our cohomology spaces into their degree subspaces,
\begin{align*}
&\mathrm{H}^3(\mathcal{W},\mathbb{K})=\bigoplus_{d\in\mathbb{Z}}\mathrm{H}^3_{(d)}(\mathcal{W},\mathbb{K})\,,\\
&\mathrm{H}^3(\mathcal{V},\mathbb{K})=\bigoplus_{d\in\mathbb{Z}}\mathrm{H}^3_{(d)}(\mathcal{V},\mathbb{K})\,.
\end{align*}
Since our module $\mathbb{K}$ is internally graded, we can use the result of Fuks (\ref{Fuks}) that the degree non-zero cohomology is cohomologically trivial. Consequently,
\begin{align*}
&\mathrm{H}^3(\mathcal{W},\mathbb{K})=\mathrm{H}^3_{(0)}(\mathcal{W},\mathbb{K})\,,\\
&\mathrm{H}^3(\mathcal{V},\mathbb{K})=\mathrm{H}^3_{(0)}(\mathcal{V},\mathbb{K})\,.
\end{align*}
Phrased differently, every cocycle is cohomologous to a degree zero cocycle.
Our algebraic Godbillon-Vey cocycle is clearly a non-trivial element in the $\mathrm{H}^3_{(0)}$-spaces.\\
Our proof of Theorem \ref{H3VKMain} shall proceed by showing that the Godbillon-Vey cocycle $\Psi$ is a generator of the spaces $\mathrm{H}^3(\mathcal{W},\mathbb{K})$ and $\mathrm{H}^3(\mathcal{V},\mathbb{K})$. Let $\psi$ be an arbitrary degree zero 3-cocycle for $\mathcal{W}$ or $\mathcal{V}$. Then we set:
\begin{equation}
\psi'=\psi-\frac{\psi (e_{-1},e_1,e_0)}{2}\ \Psi\,.\label{Comb}
\end{equation}
By (\ref{gun}), we have $\psi'(e_{-1},e_1,e_0)=0$. In the following, we will prove:
\theoremstyle{proposition}
\newtheorem{H3VK2}[GV]{Proposition}
\begin{H3VK2} \label{H3VK2}
Let $\psi$ be a 3-cocycle for $\mathcal{W}$ or  $\mathcal{V}$ with $\psi (e_{-1},e_1,e_0)=0$. \\
Then $\psi$ is a coboundary.
\end{H3VK2}
Using (\ref{Comb}) and  Proposition \ref{H3VK2}, we obtain that every cohomology class is a multiple of the algebraic Godbillon-Vey cocycle class. This implies Theorem \ref{H3VKMain}.\\
It remains to show Proposition \ref{H3VK2}. The proof of  Proposition \ref{H3VK2} will be obtained by direct, elementary algebraic manipulations. However, although we use elementary algebra to prove the last result, the proof per se is not elementary, but somewhat intricate.\\
Let $\psi$ be a 3-cochain and $\phi$ a 2-cochain for $\mathcal{W}$ or $\mathcal{V}$. These cochains will be given by their system of coefficients $\phi_{i,j},b_{i},\psi_{i,j,k},c_{i,j}\in\mathbb{K}$ defined as follows:
\begin{align}
&\psi(e_i,e_j,e_k):=\psi_{i,j,k}\qquad \text{and}\qquad \psi(e_i,e_j,t):=c_{i,j}\nonumber \,, \\
& \phi(e_i,e_j):=\phi_{i,j}\qquad\text{and}\qquad  \phi(e_i,t):=b_i \label{CoeffPhi}\,,
\end{align}
with the obvious identification coming from the alternating property of the cochains.\\
For the convenience of the calculation, we introduce the shortcut
\begin{equation*}
\alpha_i:=-\frac{1}{12}(i^3-i)\,,
\end{equation*}
for the cocycle giving the definition of the Virasoro algebra. Note that $\alpha_{-i}=-\alpha_i$, $\alpha_0=\alpha_1=\alpha_{-1}=0$ and $\alpha_2=-1/2$.\\
The cocycle condition $\delta_3\psi =0$ and the coboundary condition $\psi=\delta_2\phi$ can be reformulated in terms of the coefficients, using (\ref{CocGen}) and (\ref{CobGen}) respectively. The cochain $\psi$ is a 3-cocycle if and only if:
\begin{equation}\label{Cocycle1}
\begin{aligned}
&0=\delta_3\psi (e_i,e_j,e_k,e_l)\\
\Leftrightarrow &\ 0=(j-i)\psi_{i+j,k,l}-(k-i)\psi_{i+k,j,l}+(l-i)\psi_{i+l,j,k}+(k-j)\psi_{j+k,i,l}\\
&-(l-j)\psi_{l+j,i,k}+(l-k)\psi_{l+k,i,j}+\alpha_i\delta_{i,-j}c_{k,l}-\alpha_i\delta_{i,-k}c_{j,l}\\
&+\alpha_i\delta_{i,-l}c_{j,k}
+\alpha_j\delta_{j,-k}c_{i,l}-\alpha_j\delta_{j,-l}c_{i,k}+\alpha_k\delta_{k,-l}c_{i,j}\,,
\end{aligned}
\end{equation}
and 
\begin{equation}
0=\delta_3\psi (e_i,e_j,e_k,t)\Leftrightarrow 0=(j-i)c_{i+j,k}-(k-i)c_{i+k,j}+(k-j)c_{j+k,i}\label{Cocycle2}\,.
\end{equation}
The cochain $\psi$ is a coboundary if and only if:
\begin{equation}\label{Coboundary1}
\begin{aligned}
\psi_{i,j,k}=\delta_2\phi (e_i,e_j,e_k)&=(j-i)\phi_{i+j,k}-(k-i)\phi_{i+k,j}+(k-j)\phi_{j+k,i}\\
&-\alpha_i\delta_{i,-j}b_{k}+\alpha_i\delta_{i,-k}b_{j}-\alpha_j\delta_{j,-k}b_{i}\,,
\end{aligned}
\end{equation}
and
\begin{equation}
c_{i,j}=\delta_2\phi(e_i,e_j,t)=(j-i)b_{i+j}\label{Coboundary2} \,.
\end{equation}
As we only need to consider cochains of degree zero, and since our trivial module $\mathbb{K}$ has only degree zero elements, non-zero coefficients are only possible if the indices of the said coefficients add up to zero, hence:
\begin{equation*}
\begin{array}[h]{cll}
\psi_{i,j,k}=0&\qquad & i+j+k\neq 0\,, \\
c_{i,j}=0 &\qquad & i+j\neq 0\,, \\
\phi_{i,j}=0 &\qquad & i+j \neq 0\,, \\
b_i=0 &\qquad & i \neq 0\,.
\end{array}
\end{equation*}

Finally, we have gathered the necessary ingredients to proceed with the proof of Proposition  \ref{H3VK2}. 
The proof proceeds in several steps. In the first steps, we will fix one index of the coefficients, which we will refer to as \textit{level}, and derive results for this particular level. For the coefficients $\psi_{i,j,k}$ for example, we will first analyze $\psi_{i,-i,0}$, $\psi_{i,-i-1,1}$, $\psi_{i,-i+1,-1}$, $\psi_{i,-i-2,2}$,  and $\psi_{i,-i+2,-2}$, $\forall\ i\in\mathbb{Z}$, or coefficients with some permutation of these indices, and refer to these as coefficients of level zero, one, minus one, two and minus two respectively. In the last step, we use recurrence relations to consider general coefficients $\psi_{i,-i-k,k}$ $\forall\ i,k\in\mathbb{Z}$.


We will treat both the Witt algebra and the Virasoro algebra simultaneously. 
The proof consists of two lemmas. The first lemma mostly involves a cohomological change, although in the case of the Virasoro algebra, we will already use the cocycle condition (\ref{Cocycle2}) in order to put constraints on the coefficients $c_{i,j}$.
The second lemma involves the cocycle condition (\ref{Cocycle1}). 
\newtheorem{H3VirasoroLemma}{Lemma}[subsection]
\begin{H3VirasoroLemma} \label{lemma1}
Every 3-cocycle $\psi'\in\mathrm{H}^3(\mathcal{V},\mathbb{K})$  satisfying $\psi' (e_1,e_{-1},e_0)=0$ is cohomologous to a 3-cocycle $\psi\in\mathrm{H}^3(\mathcal{V},\mathbb{K})$ with coefficients $c_{i,j},\psi_{i,j,k}\in\mathbb{K}$ fulfilling:
\begin{equation}\label{CoeffCij}
c_{i,j}=\delta_{i,-j}\left( \frac{1}{6}\ i\ (i-1)(i+1)c_{2,-2}\right)\qquad
\text{and }\qquad \psi_{i,j,1}=0\qquad \forall\ i,j\in\mathbb{Z}\,. 
\end{equation}
Every 3-cocycle $\psi'\in\mathrm{H}^3(\mathcal{W},\mathbb{K})$ satisfying $\psi' (e_1,e_{-1},e_0)=0$ is cohomologous to a 3-cocycle $\psi\in\mathrm{H}^3(\mathcal{W},\mathbb{K})$ with coefficients $\psi_{i,j,k}\in\mathbb{K}$ fulfilling:
\begin{equation}
\psi_{i,j,1}=0\qquad \forall\ i,j\in\mathbb{Z}\,.\label{PsiIJ1}
\end{equation}
\end{H3VirasoroLemma}
\begin{proof}
Let $\psi$ be a 3-cocycle for $\mathcal{W}$ or $\mathcal{V}$ fulfilling $\psi'(e_1,e_{-1},e_0)=0$. We will perform a cohomological change $\psi=\psi ' -\delta_2\phi$ with a suitable coboundary $\delta_2\phi$ such that for the cocycle $\psi$ we have for $\mathcal{V}$ and $\mathcal{W}$ the relations (\ref{CoeffCij}) and (\ref{PsiIJ1}), respectively. Note that the coboundary is obtained by constructing a 2-cochain $\phi$ such that $\delta_2\phi$ has the desired properties. Recall that $\phi$ is given by the system of coefficients $b_0$ and $\phi_{i,-i}$. First, we consider $b_0=\phi (e_0,t)$. By setting:
\begin{equation*}
 \phi (e_0,t)=b_{0}:=\frac{c_{-1,1}'}{2}\,,
\end{equation*}
we obtain after the cohomological change:
\begin{align*}
\psi(e_{-1},e_1,t)=\psi '(e_{-1},e_1,t)-\delta_2\phi(e_{-1},e_1,t)=c_{-1,1}' -2b_0=0=c_{-1,1}\,.
\end{align*}
 Considering the cocycle condition (\ref{Cocycle2}) for $(e_i,e_{-i-1},e_{1},t)$, and taking $c_{-1,1}=0$ into account, we obtain:
\begin{equation*}
0=-(1-i)c_{i+1,-i-1}+(i+2)c_{-i,i}\,.
\end{equation*}
Hence,
\begin{equation*}
 c_{i+1,-(i+1)}=\frac{i+2}{i-1}c_{i,-i}\qquad\text{ for }i>1\,.
\end{equation*}
Starting with $i=2$ and $k>i$ we obtain:
\begin{equation*}
c_{k,-k}=\frac{1}{6} (k+1)(k)(k-1)c_{2,-2}\,.
\end{equation*}

Next, we define $\phi (e_i,e_j)$ in such a manner to obtain the announced conditions $\psi (e_i,e_{-i-1},e_1)=0$ for the cohomologous cocycle $\psi = \psi ' -\delta_2 \phi$. The coefficients $\phi_{i,-i}=\phi(e_i,e_{-i})$ can be chosen arbitrary as long as the alternating conditions $\phi_{i,-i}=-\phi_{-i,i}$ and $\phi_{0,0}=0$ are satisfied. Then $\delta_2\phi$ will be a coboundary. \\
We set $\phi_{1,-1}=\phi_{2,-2}=\phi_{0,0}=0$, and for $i\neq 1$,
\begin{equation}
\phi_{i+1,-(i+1)}:=-\frac{\psi_{i,-1-i,1}'}{(1-i)}-\frac{(2+i)}{(1-i)}\phi_{i,-i}\,,\label{RecPhi}
\end{equation}
with $\psi_{i,-1-i,1}'=\psi ' (e_i,e_{-i-1},e_1)$. Note that for $i=0$, we obtain:
\begin{equation*}
\phi_{1,-1}=-\psi_{0,-1,1}'-2\phi_{0,0}=-\psi_{0,-1,1}'=0\,,
\end{equation*}
in accordance with our general assumption. Starting from $i=2$ in (\ref{RecPhi}), taking into account $\phi_{2,-2}=0$, we obtain a definition for $\phi_{j,-j}$ with $j \geq 3$. For negative $j$, we use antisymmetry. We can rewrite (\ref{RecPhi}) as:
\begin{equation*}
\psi_{i,-1-i,1}'=(i-1)\phi_{i+1,-(1+i)}-(2+i)\phi_{i,-i}\,.
\end{equation*}
Using (\ref{Coboundary1}), we obtain:
\begin{equation*}
\delta_2\phi (e_i,e_{-i-1},e_1) =(-1-2i)\cancel{\phi_{-1,1}}+(i-1)\phi_{i+1,-1-i}+(2+i)\underbrace{\phi_{-i,i}}_{-\phi_{i,-i}}\,.
\end{equation*}
Note that the central terms vanish as either the Kronecker Delta or the $\alpha_k$ are zero. Hence:
\begin{align}
&\psi ' (e_i,e_{-i-1},e_1)-\delta_2\phi (e_i,e_{-i-1},e_1)\nonumber \\
&=\psi_{i,-1-i,1}'-(i-1)\phi_{i+1,-(1+i)}+(2+i)\phi_{i,-i}=0\,. \label{CohomChange}
\end{align}
Equation (\ref{CohomChange}) is also true for $i=0$.
\end{proof}
\theoremstyle{remark}
\newtheorem{Rem0}[R1]{Remark}
\begin{Rem0} \label{rem0}
The expression for the coefficients $c_{i,j}$ in (\ref{CoeffCij}) corresponds  to the expression of the Virasoro 2-cocycle, generator of the one-dimensional $\mathrm{H}^2(\mathcal{W},\mathbb{K})$ group; compare to (\ref{GelfandFuks}). 
The fact that the two expressions agree is not surprising, because the conditions  (\ref{Coboundary2}) and (\ref{Cocycle2}) involving the generator $t$ reduce to the coboundary and cocycle conditions of $\mathrm{H}^2(\mathcal{W},\mathbb{K})$. 
\end{Rem0}

In the second lemma, we will use the cocycle condition (\ref{Cocycle1}).
\theoremstyle{theorem}
\newtheorem{H3VirasoroLemma2}[H3VirasoroLemma]{Lemma}
\begin{H3VirasoroLemma2} \label{lemma2}
The following holds:
\begin{enumerate}
\item[(i)] Let $\psi\in\mathrm{H}^3(\mathcal{V},\mathbb{K})$ be a $3$-cocycle such that:
\begin{align*}
&c_{i,j}=\delta_{i,-j}\left( \frac{1}{6}(i-1)(i)(i+1)c_{2,-2}\right)\qquad 
\text{and }\qquad\ \psi_{i,j,1}=0\qquad \forall\ i,j\in\mathbb{Z}\,.\\
\text{Then}\qquad& c_{i,j}=0\qquad \forall\ i,j\in\mathbb{Z} \qquad
\text{ and }\qquad \psi_{i,j,k}=0\qquad \forall \ i,j,k\in\mathbb{Z}\,.
\end{align*}
\item[(ii)]  Let $\psi\in\mathrm{H}^3(\mathcal{W},\mathbb{K})$ be a $3$-cocycle such that: 
\begin{align*}
&\psi_{i,j,1}=0\qquad \forall\ i,j\in\mathbb{Z}\,.\\
\text{Then}\qquad &\psi_{i,j,k}=0\qquad \forall \ i,j,k\in\mathbb{Z}\,.
\end{align*}
\end{enumerate}
\end{H3VirasoroLemma2}
\begin{proof}
We write down the proof in the setting of the Virasoro algebra.
However, the central terms cancel in most of the cocycle conditions which we will consider in the following. Hence, the conclusions obtained are valid for both the Witt algebra and the Virasoro algebra.
The central terms only appear towards the end of the proof, where we will explicitly point out the differences between the proof for the Witt algebra and for the Virasoro algebra.\\
The coefficients to consider are of the form $\psi_{i,j,k}$ with $i+j+k=0$, and can be written as $\psi_{-i-k,i,k}$. We will first show that they are generated  by a single coefficient $\psi_{-2,2,0}$, and subsequently we will show  that this coefficient is zero.\\
We will proceed by constructing recurrence relations. In order to have recurrence relations, at least one of the generators $e_i$ appearing in the cocycle condition (\ref{Cocycle1}) is taken to be of degree plus or minus one, $e_i=e_{\pm 1}$. This means that no coefficient $c_{i,j}$ will appear in the recurrence relations, because either the coefficient $c_{i,j}$ is equal to a coefficient of the form $c_{i,\pm 1}$, which is zero due to our assumption, or its pre-factor of the form $(i^3-i)\delta_{i,j}$ is zero because $((\pm 1)^3\mp 1)=0$. Therefore, we can once again treat the Virasoro algebra and the Witt algebra simultaneously.\\
We will start with level zero, $k=0$. The cocycle condition (\ref{Cocycle1}) on the generators $(e_{-i-1},e_i,e_0,e_1)$ yields: 
\begin{align}
&(1+2 i) \cancel{\psi_{-1,0,1}}+\cancel{\psi_{1,-1-i,i}}-\cancel{(1+i) \psi_{-1-i,i,1}}\nonumber\\
&+(2+i) \psi_{-i,i,0}-i \cancel{\psi_{i,-1-i,1}}+(-1+i) \psi_{1+i,-1-i,0}=0\nonumber\\
&\Leftrightarrow \psi_{-1-i,1+i,0}=\frac{(2+i)}{(-1+i)} \psi_{-i,i,0}\,.\label{RecLevelZero}
\end{align}
The first term is zero by assumption. The other slashed terms cancel each other, and are zero anyway by assumption. Starting with $i=2$, increasing $i$, we see that level zero is generated solely by the coefficient $\psi_{-2,2,0}$.  Values of $i<2$ do not yield any new information.\\
Let us continue with level minus one, $k=-1$. A recurrence relation for level minus one is obtained by considering the cocycle condition (\ref{Cocycle1}) on the generators $(e_{-i},e_i,e_{-1},e_1)$:
\begin{align}
&2 i \cancel{\psi_{0,-1,1}}+2 \psi_{0,-i,i}-(-1+i) \cancel{\psi_{-1-i,i,1}}+(1+i) \psi_{1-i,i,-1}\nonumber\\
&-(1+i) \cancel{\psi_{-1+i,-i,1}}+(-1+i) \psi_{1+i,-i,-1}=0\nonumber\\
&\Leftrightarrow \psi_{1-i,i,-1}=\frac{1}{(1+i)}(-2 \psi_{0,-i,i}-(-1+i) \psi_{1+i,-i,-1})\,.\label{RecLevelMinusOne}
\end{align}
The slashed terms are zero by assumption. Starting with $i=-2$, decreasing $i$, we see that level minus one is generated by the same generating coefficient as level zero, namely $\psi_{-2,2,0}$. In other words, we have no new generator of level minus one appearing, because at the starting point of the recurrence, $i=-2$, the starting coefficient $\psi_{-1,2,-1}$ is zero due to the alternating property.
Values of $i$ with $i>-2$ do not lead to any new information. \\
We can proceed similarly with level minus two $k=-2$. The cocycle condition (\ref{Cocycle1}) for the generators $(e_{-i+1},e_i,e_{-2},e_1)$ gives us the following recurrence relation:
\begin{align}
 3 \psi_{-1,1-i,i}&-(-3+i) \cancel{\psi_{-1-i,i,1}}+i \psi_{2-i,i,-2}\nonumber\\
&-(2+i)  \cancel{\psi_{-2+i,1-i,1}}+(-1+i) \psi_{1+i,1-i,-2}=0\nonumber\\
\Leftrightarrow \psi_{2-i,i,-2}&=\frac{1}{i}(-3 \psi_{-1,1-i,i}-(-1+i) \psi_{1+i,1-i,-2})\,.\label{RecLevelMinusTwo}
\end{align}
In the equations above, we already dropped the coefficient $\psi_{i+j,k,l}$ which is trivially zero for this combination of generators $e_i$ due to the alternating property. 
The slashed terms are of level plus one and are thus zero due to our assumption.
Starting with $i=-3$, $i$ decreasing, we see that also level minus two is determined entirely by $\psi_{-2,2,0}$. In fact, no new generator of level minus two appears because at $i=-3$, the starting coefficient $\psi_{-2,4,-2}$ vanishes because of the alternating property.
Again, taking $i>-3$ does not lead to any new information.\\
The same can be done for level plus two $k=2$, by considering the cocycle condition (\ref{Cocycle1}) for the generators $(e_{-i-1},e_i,e_{2},e_{-1})$:
\begin{align}
-3 \cancel{\psi_{1,-1-i,i}}&+i \psi_{-2-i,i,2}-(3+i) \psi_{1-i,i,-1}\nonumber\\
&+(1+i) \psi_{-1+i,-1-i,2}-(-2+i) \psi_{2+i,-1-i,-1}=0\nonumber\\
\Leftrightarrow \psi_{-2-i,i,2}&=\frac{1}{i}((3+i) \psi_{1-i,i,-1}-(1+i) \psi_{-1+i,-1-i,2}+(-2+i) \psi_{2+i,-1-i,-1})\,.\label{RecLevelPlusTwo}
\end{align}
Again, we dropped the term $\psi_{i+j,k,l}$ because this coefficient is trivially zero for the choice of generators $e_i$ under consideration.
The slashed term is zero by assumption. 
Starting with $i=3$, $i$ increasing, it is clear that also level plus two is solely generated by $\psi_{-2,2,0}$. Once more, taking $i<3$ does not yield any new information.\\
Finally, we can produce recurrence relations for general level $k$. Starting with positive $k$, the cocycle condition (\ref{Cocycle1}) for the generators $(e_{-i-k-1},e_i,e_{k},e_{1})$ yields:
\begin{align}
&-(1+i+2 k) \cancel{\psi_{-1-i,i,1}}+(-1+i) \psi_{1+i,-1-i-k,k}+(1+2 i+k) \cancel{\psi_{-1-k,k,1}}\nonumber\\
&+(2+i+k) \psi_{-i-k,i,k}-(-1+k) \psi_{1+k,-1-i-k,i}+(-i+k) \cancel{\psi_{i+k,-1-i-k,1}}=0\nonumber\\
&\Leftrightarrow \psi_{-1-i-k,i,1+k}=\frac{1}{1-k}(-(-1+i) \psi_{1+i,-1-i-k,k}-(2+i+k) \psi_{-i-k,i,k})\,.\label{RecLevelPlusK}
\end{align}
The slashed terms are of level plus one and thus zero by assumption.
Starting with $k=2$ and $i=k+2$, increasing $k$ and $i$, we see that at level $k+1$, no new generators of level $k+1$ appear, since on the right-hand side of the recurrence relation, only coefficients of level $k$ appear.
Therefore, a general positive level $k$ is build from the generator $\psi_{-2,2,0}$. Considering $i<k+2$ does not yield new information due to the alternating property.\\
The same can be done for negative $k$ by considering the cocycle condition (\ref{Cocycle1}) for the generators $(e_{-i-k+1},e_i,e_{k},e_{-1})$:
\begin{subequations}\label{RecLevelMinusK}
\begin{align}
&-(-1+i+2 k) \psi_{1-i,i,-1}+(1+i) \psi_{-1+i,1-i-k,k}+(-1+2 i+k) \psi_{1-k,k,-1}\nonumber\\
&+(-2+i+k) \psi_{-i-k,i,k}-(1+k) \psi_{-1+k,1-i-k,i}+(-i+k) \psi_{i+k,1-i-k,-1}=0\nonumber\\
& \Leftrightarrow \psi_{1-i-k,i,-1+k}=\frac{1}{-1-k}((-1+i+2 k) \psi_{1-i,i,-1}-(1+i) \psi_{-1+i,1-i-k,k}\\
&-(-1+2 i+k) \psi_{1-k,k,-1}-(-2+i+k) \psi_{-i-k,i,k}-(-i+k) \psi_{i+k,1-i-k,-1})\,.
\end{align}
\end{subequations}
Starting with $k=-2$ and $i=-2+k$, decreasing $i$ and $k$, we see that also a general negative level $k$ is build solely from the generating coefficient  $\psi_{-2,2,0}$. Indeed, at level $k-1$, no new generator of level $k-1$ appears, since the right-hand side of the recurrence relation only contains coefficients of level $k$ and minus one. 
Again, taking $i>-2+k$ does not yield new information.\\

Let us summarize the results we obtained so far. We showed that the coefficients $\psi_{i,j,k}$ are solely determined by the generating coefficient  $\psi_{-2,2,0}$. In the analysis above, no central terms appeared and thus, the conclusion is valid for both the Witt and the Virasoro algebra. Moreover, in the case of the Virasoro algebra, we showed in the previous lemma that the coefficients $c_{i,j}$ are generated by a single coefficient, namely $c_{-2,2}$. \\
In the last step of the proof, we have to check whether there are non-trivial relations between the two generators $c_{-2,2}$ and $\psi_{-2,2,0}$, which could force them to zero. Once we proved that the generating coefficients are zero, it follows that all the other coefficients $\psi_{i,j,k}$ and $c_{i,j}$ are also zero, due to the recurrence relations.\\
In order to prove that the generating coefficients are zero, we need to find at least two non-trivial relations in the case of the Virasoro algebra, and one non-trivial relation in the case of the Witt algebra. We will consider the cocycle condition (\ref{Cocycle1}) for the generators $(e_{-4},e_{-3},e_2,e_5)$ and $(e_{-3},e_{-2},e_2,e_3)$. The cocycle condition for $(e_{-4},e_{-3},e_2,e_5)$ yields:
\begin{align}\label{Coc1}
&\psi_{-7,2,5}-6 \psi_{-2,-3,5}+5 \psi_{-1,-4,5}+9 \underbrace{\psi_{1,-3,2}}_{=0}+3 \psi_{7,-4,-3}=0\nonumber\\
&\Leftrightarrow -\psi_{-7,5,2}+6 \psi_{5,-3,-2}-5 \psi_{5,-4,-1}+3 \psi_{7,-4,-3}=0\,,
\end{align}
whereas the one for $(e_{-3},e_{-2},e_2,e_3)$ yields:
\begin{align}\label{Coc2}
&\frac{1}{2} c_{-3,3}+2 c_{-2,2}+\psi_{-5,2,3}-5 \psi_{-1,-2,3}+4 \psi_{0,-3,3}+6 \psi_{0,-2,2}-5 \underbrace{\psi_{1,-3,2}}_{=0}+\psi_{5,-3,-2}=0\nonumber\\
&\Leftrightarrow\frac{1}{2} c_{-3,3}+2 c_{-2,2}-\psi_{-5,3,2}+5 \psi_{3,-2,-1}+4 \psi_{0,-3,3}+6 \psi_{0,-2,2}+\psi_{5,-3,-2}=0\,.
\end{align}
The terms of level plus one are zero by assumption. We will use the recurrence relations (\ref{CoeffCij}), (\ref{RecLevelZero}), (\ref{RecLevelMinusOne}), (\ref{RecLevelMinusTwo}), (\ref{RecLevelPlusTwo}), (\ref{RecLevelPlusK}) and (\ref{RecLevelMinusK}) to express all the coefficients $\psi_{i,j,k}$ and $c_{i,j}$ appearing in the conditions above in terms of the generators $c_{-2,2}$ and $\psi_{-2,2,0}$. We will write down all the coefficients which are needed, be it implicitly or explicitly, in order to expose the structure of the recurrence relations and their entanglement. We will see that the cocycle condition (\ref{Coc1}) for $(e_{-4},e_{-3},e_2,e_5)$ yields an non-trivial relation for $\psi_{-2,2,0}$. As no central terms appear, the cocycle condition (\ref{Coc1}) is valid both for the Witt and the Virasoro algebra. In case of the Witt algebra, the cocycle condition (\ref{Coc1}) will be sufficient to conclude.
In case of the Virasoro algebra, the cocycle condition (\ref{Coc1}), together with the second cocycle condition (\ref{Coc2}) for $(e_{-3},e_{-2},e_2,e_3)$, will yield $c_{i,j}=0\ \forall\ i,j\in\mathbb{Z}$, which allows to conclude.\\
Let us begin with the coefficients of level zero. The recurrence relation (\ref{RecLevelZero}) yields for $i=2$ the following expression for $\psi_{-3,3,0}$:
\begin{equation}
\psi_{-3,3,0}=4\psi_{-2,2,0}\,. \label{PsiM330}
\end{equation}
Continuing with $i=3,4,5$ we obtain respectively:
\begin{equation}
\psi_{-4,4,0}= \frac{5}{2}\psi_{-3,3,0}
\stackrel{(\ref{PsiM330})}{\Leftrightarrow} \psi_{-4,4,0}=10\psi_{-2,2,0}\label{PsiM440}\,,
\end{equation}
and 
\begin{equation}
\psi_{-5,5,0}= 2\psi_{-4,4,0}
\stackrel{(\ref{PsiM440})}{\Leftrightarrow} \psi_{-5,5,0}=20\psi_{-2,2,0}\label{PsiM550}\,,
\end{equation}
and 
\begin{equation}
\psi_{-6,6,0}= \frac{7}{4}\psi_{-5,5,0}
\stackrel{(\ref{PsiM550})}{\Leftrightarrow} \psi_{-6,6,0}=35\psi_{-2,2,0}\label{PsiM660}\,.
\end{equation}
More coefficients of level zero will not be needed. Hence, let us consider the coefficients of level minus one. \\
The recurrence relation (\ref{RecLevelMinusOne}) yields for $i=-2,-3,-4,-5,-6$ the following coefficients, i.e. starting with $i=-2$,
\begin{equation}
\psi_{3,-2,-1}= -(-2\psi_{0,2,-2}+3\cancel{\psi_{-1,2,-1}})
\Leftrightarrow\psi_{3,-2,-1}= -2\psi_{-2,2,0}\label{Psi3M2M1}\,,
\end{equation}
and for $i=-3$,
\begin{equation}
\psi_{4,-3,-1}= -\frac{1}{2}(-2\psi_{0,3,-3}+4\psi_{-2,3,-1})
\stackrel{(\ref{Psi3M2M1}),(\ref{PsiM330})}{\Leftrightarrow}\psi_{4,-3,-1}=-8\psi_{-2,2,0}\label{Psi4M3M1}\,,
\end{equation}
for $i=-4$,
\begin{equation}
\psi_{5,-4,-1}= -\frac{1}{3}(-2\psi_{0,4,-4}+5\psi_{-3,4,-1})
\stackrel{(\ref{Psi4M3M1}),(\ref{PsiM440})}{\Leftrightarrow}\psi_{5,-4,-1}=-20\psi_{-2,2,0}\label{Psi5M4M1}\,,
\end{equation}
for $i=-5$,
\begin{equation}
\psi_{6,-5,-1}= -\frac{1}{4}(-2\psi_{0,5,-5}+6\psi_{-4,5,-1})
\stackrel{(\ref{Psi5M4M1}),(\ref{PsiM550})}{\Leftrightarrow}\psi_{6,-5,-1}=-40\psi_{-2,2,0}\label{Psi6M5M1}\,,
\end{equation}
and finally for $i=-6$,
\begin{equation}
\psi_{7,-6,-1}= -\frac{1}{5}(-2\psi_{0,6,-6}+7\psi_{-5,6,-1})
\stackrel{(\ref{Psi6M5M1}),(\ref{PsiM660})}{\Leftrightarrow}\psi_{7,-6,-1}=-70\psi_{-2,2,0}\label{Psi7M6M1}\,.
\end{equation}
More coefficients of level minus one will not be needed. We continue with the coefficients of level plus two.\\
The recurrence relation (\ref{RecLevelPlusTwo}) yields for $i=3,4,5$ the following coefficients, starting with $i=3$:
\begin{equation}
\psi_{-5,3,2}= \frac{1}{3}(6\psi_{-2,3,-1}-4\cancel{\psi_{2,-4,2}}+\psi_{5,-4,-1})
\stackrel{(\ref{Psi3M2M1}),(\ref{Psi5M4M1})}{\Leftrightarrow}\psi_{-5,3,2}=-\frac{8}{3}\psi_{-2,2,0}\,,\label{PsiM532}
\end{equation}
for $i=4$,
\begin{equation}
\psi_{-6,4,2}= \frac{1}{4}(7\psi_{-3,4,-1}-5\psi_{3,-5,2}+2\psi_{6,-5,-1})
\stackrel{(\ref{Psi4M3M1}),(\ref{Psi6M5M1}),(\ref{PsiM532})}{\Leftrightarrow}\psi_{-6,4,2}=-\frac{28}{3}\psi_{-2,2,0}\,,\label{PsiM642}
\end{equation}
and finally for $i=5$,
\begin{equation}
\psi_{-7,5,2}= \frac{1}{5}(8\psi_{-4,5,-1}-6\psi_{4,-6,2}+3\psi_{7,-6,-1})
\stackrel{(\ref{Psi5M4M1}),(\ref{Psi7M6M1}),(\ref{PsiM642})}{\Leftrightarrow}\psi_{-7,5,2}=-\frac{106}{5}\psi_{-2,2,0}\,.\label{PsiM752}
\end{equation}
More coefficients of level plus two are not needed. We will continue with the coefficients of level minus two.\\
The recurrence relation (\ref{RecLevelMinusTwo}) yields for $i=-3,-4,-5$ the following coefficients, starting with $i=-3$:
\begin{equation}
\psi_{5,-3,-2}= -\frac{1}{3}(-3\psi_{-1,4,-3}+4\cancel{\psi_{-2,4,-2}})
\stackrel{(\ref{Psi4M3M1})}{\Leftrightarrow}\psi_{5,-3,-2}=-8\psi_{-2,2,0}\,,\label{Psi5M3M2}
\end{equation}
for $i=-4$ we obtain
\begin{equation}
\psi_{6,-4,-2}= -\frac{1}{4}(-3\psi_{-1,5,-4}+5\psi_{-3,5,-2})
\stackrel{(\ref{Psi5M4M1}),(\ref{Psi5M3M2})}{\Leftrightarrow}\psi_{6,-4,-2}=-25\psi_{-2,2,0}\,,\label{Psi6M4M2}
\end{equation}
and for $i=-5$
 \begin{equation}
\psi_{7,-5,-2}= -\frac{1}{5}(-3\psi_{-1,6,-5}+6\psi_{-4,6,-2})
\stackrel{(\ref{Psi6M5M1}),(\ref{Psi6M4M2})}{\Leftrightarrow}\psi_{7,-5,-2}=-54\psi_{-2,2,0}\,.\label{Psi7M5M2}
\end{equation}
These are all the coefficients needed for level minus two. Next, we need some coefficients for level minus three.\\
Putting $k=-2$ in the recurrence relation (\ref{RecLevelMinusK}), we obtain a recurrence relation for level minus three:
\begin{subequations}\label{RecLevelMinus3}
\begin{align}
\psi_{3-i,i,-3}=&(-5+i)\psi_{1-i,i,-1}-(1+i)\psi_{-1+i,3-i,-2}-(-3+2i)\psi_{3,-2,-1}\\
&-(-4+i)\psi_{2-i,i,-2}-(-i-2)\psi_{i-2,3-i,-1}\,.
\end{align}
\end{subequations}
We need only one coefficient of level minus three, namely the coefficient obtained by taking $i=-4$ in (\ref{RecLevelMinus3}):
\begin{align}
&\psi_{7,-4,-3}=-9\psi_{5,-4,-1}+3\psi_{-5,7,-2}+11\psi_{3,-2,-1}+8\psi_{6,-4,-2}-2\psi_{-6,7,-1}\nonumber\\
&\stackrel{(\ref{Psi5M4M1}),(\ref{Psi7M5M2}),(\ref{Psi3M2M1}),(\ref{Psi6M4M2}),(\ref{Psi7M6M1})}{\Leftrightarrow}\psi_{7,-4,-3}=-20\psi_{-2,2,0}\,.\label{Psi7M4M3}
\end{align}
These recurrence relations are valid both for the Witt algebra and the Virasoro algebra, since no central terms appear.\\
At last, we will need the coefficient $c_{-3,3}$ in the case of the Virasoro algebra. Taking the relation (\ref{CoeffCij}) and putting $i=-3$, $j=3$, we obtain:
\begin{equation}
c_{-3,3}=4\ c_{-2,2}\,.\label{CM33}
\end{equation}
Finally, we obtained all the coefficients needed. Inserting the coefficients (\ref{PsiM752}), (\ref{Psi5M3M2}), (\ref{Psi5M4M1}) and  (\ref{Psi7M4M3}) into the cocycle condition (\ref{Coc1}), we obtain:
\begin{align}
&\frac{106}{5}\psi_{-2,2,0}-48\psi_{-2,2,0}
+100\psi_{-2,2,0}-60\psi_{-2,2,0}=0\nonumber\\
&\Leftrightarrow \psi_{-2,2,0}=0\,.\label{Final1}
\end{align}
This already allows to conclude for the Witt algebra.
Similarly, inserting the coefficients (\ref{PsiM532}), (\ref{Psi3M2M1}), (\ref{PsiM330}), (\ref{Psi5M3M2}) and (\ref{CM33}) into the cocycle condition (\ref{Coc2}), we obtain:
\begin{align}
&\frac{8}{3}\psi_{-2,2,0}-10\psi_{-2,2,0}
+16\psi_{-2,2,0}+6\psi_{-2,2,0}-8\psi_{-2,2,0}+4c_{-2,2}=0\nonumber\\
&\Leftrightarrow 3 c_{-2,2}+5\psi_{-2,2,0}=0\,.\label{Final2}
\end{align}
Equations (\ref{Final1}) and (\ref{Final2}) together yield $c_{-2,2}=0$ and $\psi_{-2,2,0}=0$, yielding the conclusion also for the Virasoro algebra.
\end{proof}

\textbf{Proof of Proposition \ref{H3VK2}:}
\begin{proof}
Lemma \ref{lemma1} tells us that we can always perform a cohomological change of a 3-cocycle in $\mathrm{H}^3(\mathcal{W},\mathbb{K})$ or $\mathrm{H}^3(\mathcal{V},\mathbb{K})$  satisfying $\psi (e_{-1},e_1,e_0)=0$ such that we obtain a cohomological equivalent 3-cocycle with coefficients fulfilling the assumptions of Lemma \ref{lemma2}. Lemma \ref{lemma2} tells us that all the coefficients of the 3-cocycle satisfying $\psi (e_{-1},e_1,e_0)=0$ are zero. This allows to conclude. 
\end{proof}
\theoremstyle{remark}
\newtheorem{R2}[R1]{Remark}
\begin{R2}
In the proof of Proposition \ref{H3VK2}, $\mathcal{W}$ and  $\mathcal{V}$ were treated in a similar manner, finally yielding $\mathrm{H}^3(\mathcal{W},\mathbb{K})\cong \mathrm{H}^3(\mathcal{V},\mathbb{K})$. The reader might have gotten the impression that this could also be true for other cohomology spaces. However, already 
\begin{equation}
1=dim(\mathrm{H}^2(\mathcal{W},\mathbb{K}))\neq dim(\mathrm{H}^2(\mathcal{V},\mathbb{K}))=0\,,
\end{equation}
shows that the underlying structure is more delicate. The Hochschild-Serre spectral sequence with the trivial module $M=\mathbb{K}$ in (\ref{HochSerre}) gives more information.
\end{R2}



\providecommand{\bysame}{\leavevmode\hbox to3em{\hrulefill}\thinspace}
\providecommand{\MR}{\relax\ifhmode\unskip\space\fi MR }
\providecommand{\MRhref}[2]{%
  \href{http://www.ams.org/mathscinet-getitem?mr=#1}{#2}
}
\providecommand{\href}[2]{#2}

\end{document}